\documentclass[a4paper,12pt]{amsart}
\usepackage{enumitem}
\usepackage{geometry}

\begin{document}

\title[On radial positive normalized solutions of NLS]
        {On radial positive normalized solutions of the Nonlinear Schrödinger equation in an annulus}
\author{Jian Liang$^{\ast}$}
\address[Jian Liang, Linjie Song]{Institute of Mathematics, AMSS\\ 
        Academia Sinica\\ 
        Beijing 100190, China}
\address[Jian Liang, Linjie Song]{University of Chinese Academy of Science\\
        Beijing 100049, China}
\email[Corresponding author]{liangjian2020@amss.ac.cn}
\thanks{$^{\ast}$ Corresponding author}

\author{Linjie Song}
\email{songlinjie18@mails.ucas.edu.cn}
\thanks{Linjie Song is supported by CEMS}

\keywords{Normalized solutions, orbital stability, nonlinear Schrödinger equations, annulus}
\subjclass[2020]{35A15, 35B35, 35B38, 35J20, 35J10, 35Q55, 35C08}

\theoremstyle{plain}
\newtheorem{theorem}{Theorem}[section]
\newtheorem{lemma}[theorem]{Lemma}
\newtheorem{corollary}[theorem]{Corollary}
\newtheorem{proposition}[theorem]{Proposition}

\theoremstyle{remark}
\newtheorem{remark}[theorem]{Remark}
\newtheorem*{notation}{Notation}

\numberwithin{equation}{section}

\bibliographystyle{amsplain}

    \begin{abstract}
        We are interested in the following semilinear elliptic problem:
        \begin{equation*} 
            \begin{cases}
            -\Delta u + \lambda u = u^{p-1}, x \in T, \\
            u > 0, u = 0 \ \text{on} \ \partial T, \\
            \int_{T}u^{2} \, dx= c
            \end{cases}
        \end{equation*}
        where $T = \{x \in \mathbb{R}^{N}: 1 < |x| < 2\}$ is an annulus in $\mathbb{R}^{N}$, $N \geq 2$, $p > 1$ is Sobolev-subcritical, searching for
        conditions (about $c$, $N$ and $p$) for the existence of positive radial solutions. We analyze the asymptotic behavior of $c$ as $\lambda \rightarrow +\infty$ and
        $\lambda \rightarrow -\lambda_1$ to get the existence, non-existence and multiplicity of normalized solutions. Additionally, based on the properties of these solutions,
        we extend the results obtained in Pierotti et al. in Calculus of Variations and Partial Differential Equations, 2017, 56: 1-27. In contrast of the earlier results, a positive radial solution with arbitrarily large mass
        can be obtained when $N \geq 3$ or if $N = 2$ and $p < 6$. Our paper also includes the demonstration of orbital stability/instability results.
    \end{abstract}
    
    \maketitle

\section{Introduction}
    Consider the following problem:
    \begin{equation} \label{NLS}
        \begin{cases}
        -\Delta u + \lambda u = |u|^{p-2}u, \  x \in T, \\
        u = 0 \ \text{on} \ \partial T.
        \end{cases}
    \end{equation}
    where $T = \{x \in \mathbb{R}^{N}: 1 < |x| < 2\}$ is an annulus in $\mathbb{R}^{N}$, $N \geq 2$, $2 < p < 2^{\ast}$ where
    $2^{\ast} := \frac{2N}{N-2}$ if $N \geq 3$ and $2^{\ast} := +\infty$ if $N = 2$.

    Solutions of \eqref{NLS} come from establishing the existence of standing wave solutions  $\Phi(x,t) = e^{i\lambda t} u(x)$ of the following nonlinear Schrödinger equation (NLSE):
    \begin{equation}\label{NLSE}
        \begin{cases}
            i\frac{\partial \Phi}{\partial t} + \Delta \Phi + |\Phi|^{p-2}\Phi = 0,  (t,x) \in  \mathbb{R} \times T \\
            \Phi(0,x) = \Phi_0(x) 
        \end{cases}
    \end{equation}

    When seeking solutions to \eqref{NLS}, two methods are used: the first involves fixing parameter $\lambda$ and subsequently searching for solutions as critical points of the \textit{action} functional 
    \begin{equation*}
    \Phi_{\lambda}(u) := \frac{1}{2}\int_{T} (|\nabla u|^{2} + \lambda u^{2}) - \frac{1}{p}\int_{T}|u|^{p}.
    \end{equation*}

    An alternative method is to search for solutions with a prescribed mass, where the Lagrange multiplier $\lambda \in \mathbb{R}$ is unknown. The primary objective of this 
    paper is to prove the existence of positive \textit{normalized solutions} of \eqref{NLS}, which satisfy 
    \begin{equation} 
        \int_{T}u^{2} \, dx= c
    \end{equation}
    where $c > 0$ is prescribed. 

    The mass critical exponent, denoted by $p_{c} = 2 + \frac{4}{N}$, has played a crucial role in analyzing normalized solutions. It distinguishes two distinct regimes: the mass-subcritical regime for $2 < p < 2 + \frac{4}{N}$ and the mass-supercritical regime 
    for $2 + \frac{4}{N} < p < 2^{\ast}$. This exponent serves as a threshold for several dynamic properties, including the stability or instability of ground states and the question of global existence versus blow-up (see \cite{cazenave1982orbital} \cite{glassey1977blowing} \cite{berestycki1983nonlinear}).
    
    There are many results dealing with the normalized solutions on $\mathbb{R}^{N}$. For example, see \cite{jeanjean1997existence,bartsch2013normalized,bartsch2017natural,soave2020normalized,soave2020normalizedcritical,bieganowski2021normalized,bartsch2021normalized} and their references. For the normalized solutions on a bounded domain, 
    to our knowledge, the only papers are \cite{noris2015existence,pierotti2017normalized,noris2019normalized,hajaiej2022general,song2022existence}. 
    
    In \cite{noris2015existence}, the authors dealt with problem \eqref{NLS} in the spherical domain $\Omega = B_1$. It is shown that the solvability of \eqref{NLS} is 
    strongly influenced by the exponent $p$. In \cite{pierotti2017normalized}, the authors investigate problem \eqref{NLS} in the context of a general $C^1$ bounded domain 
    $\Omega \subset \mathbb{R}^{N}$, with a focus on finding solutions that have higher Morse index. It is shown that if we define 
    \begin{equation*}
        \mathfrak{A}_k = \mathfrak{A}_k(p, \Omega) = 
        \begin{Bmatrix}
            c > 0 : \eqref{NLS} \text{ admits a solution }u \text{ (for some $\lambda$)} \\
            \text{ having Morse index } m(u) \leq k
        \end{Bmatrix}
    \end{equation*}
    where 
    \begin{equation*}
            m(u) = \max 
            \begin{Bmatrix}
                k: \exists V \subset H_{0}^{1}(\Omega), \dim V = k, \forall v \in V \backslash \{0\}, \\
                  \int_{\Omega} |\nabla v|^2 + \lambda v^2 - p|u|^{p-1}v^2 \, dx < 0 
            \end{Bmatrix}
            \in \mathbb{N}
    \end{equation*}
    is the Morse index of $u$ in $H_{0}^{1}(\Omega)$. If we denote $\mu_k = \sup \mathfrak{A}_k (p, \Omega)$ then if $k$ is fixed, we have 
    \begin{equation*}
        \mu_k < +\infty \Longleftrightarrow p > 2+ \frac{4}{N}.
    \end{equation*}
    The paper provided a detailed blow-up analysis, supported by appropriate a priori pointwise estimates, to prove this result. Specifically, the analysis concerns a 
    sequence of solutions $\{u_n\}$ for \eqref{NLS}, where $\|u_n\|_{H_0^1(\Omega)} \rightarrow +\infty$ and there exists $\bar{k} \in \mathbb{N}$ such that $m(u_n) \leq \bar{k}$.
    As $n \rightarrow \infty$, the solution will blow up on a finite number of local maximas; this number is at most $\bar{k}$, and the points are far away from each other and 
    the boundary. Moreover, these solutions will also decay exponentially away from the blow-up points.

    As for positive solutions of \eqref{NLS}, the authors proved the following result:

    \medskip

    For every $0 < c < c^1 = c^1(\Omega, p)$ problem \eqref{NLS} admits a solution which is a local minima of the energy $E$ on $\mathcal{M}_c$ and has Morse index one. Where   
    \begin{equation*}
    E(u) = \frac{1}{2}\int_{T} |\nabla u|^{2} - \frac{1}{p}\int_{T}|u|^{p}
    \end{equation*}
    is the \textit{energy} functional and $\mathcal{M}_c = \{ u \in H_0^1(\Omega), \|u_n\|_{L^2}^2 = c\}$ is the $L^2$ constraint sphere.
    \begin{itemize}
        \item If $2 < p < 2 + \frac{4}{N}$,  $c^1(\Omega, p) = +\infty$;
        \item If $p = 2 + \frac{4}{N}$, $c^1(\Omega, p) > \| Z_{N,p}\|^2_{L^2(\mathbb{R}^N)}$;
        \item If $2 + \frac{4}{N} < p < 2^{\ast}$, $c^1(\Omega, p) > D_{N,p} \lambda_1(\Omega)^{\frac{2}{p-1}-\frac{N}{2}}$ 
        where $D_{N,p}$ is a universal constant.
    \end{itemize}

    In the case of a symmetric domain, any solution can be used as a building block to create other solutions with more complex behavior. This process leads to the creation of 
    the so-called 'necklace solitary waves' . In paper \cite{pierotti2017normalized}, the authors employed this method to demonstrate that if $\Omega = B$ is a ball in $\mathbb{R}^{N}$, then 
    \begin{equation*}
        p < 2 + \frac{4}{N-1}  \Longrightarrow \eqref{NLS} \text{ admits a solution for every }c > 0.
    \end{equation*}
    When $\Omega = R$ is a rectangle, the further restriction on $p$ can be removed. In fact, it means that $\mu_k \rightarrow +\infty$ as $k \rightarrow \infty$ in these cases.

    For an annular domain, we can focus on positive radial solutions of \eqref{NLS} due to the symmetry of $T$. In a series of work (\cite{hajaiej2022general,song2022properties,song2022existence,song2022threshold,song2022new,hajaiej2023strict}), the second author with his co-author developed a completely novel framework to prove the existence of normalized solutions in 
    various cases. Paper \cite{song2022threshold} specifically delved into problem $\eqref{NLS}$ in the domain outside of the ball $D = \{x \in \mathbb{R}^{N}: |x| > 1\}$ by analyzing the 
    asymptotic behaviors of $u$ in cases where $\lambda \rightarrow 0^+$ and $\lambda \rightarrow +\infty$. It is shown that
    \begin{enumerate}[leftmargin = 2em, itemindent = 0pt]
        \item If $N \geq 3$ and $2 + \frac{4}{N} \leq p < 2^{\ast}$  or if $N = 2$ and $4 < p < 6$, there exists $\eta_{1} > 0$ such that: 
        \begin{itemize}
            \item For any $c > \eta_{1}$, \eqref{NLS} admits at least two solutions $u_{\lambda}$ and $u_{\widetilde{\lambda}}$ in $H_{0,rad}^{1}(D)$ with $\lambda > \widetilde{\lambda} > 0$, s.t. $\| u_{\lambda}\|_{L^{2}(D)} = \| u_{\widetilde{\lambda}}\|_{L^{2}(D)} = c$;
            \item \eqref{NLS} admits at least one solution $u_{\lambda}$ in $H_{0,rad}^{1}(D)$ with some $\lambda > 0$, s.t. $\| u_{\lambda}\|_{L^{2}(D)}=\eta_{1}$;
            \item For any $c < \eta_{1}$, \eqref{NLS} admits no solution $u_{\lambda}$ in $H_{0,rad}^{1}(D)$, s.t. $\| u_{\lambda}\|_{L^{2}(D)} = c$.
        \end{itemize}
        \item If $N = 2$ and $p = 6$, there exists $\eta_{2} > 0$ such that:
        \begin{itemize}
            \item For any $c > \eta_{2}$, \eqref{NLS} admits at least one solution $u_{\lambda}$ in $H_{0,rad}^{1}(D)$ with some $\lambda > 0$, s.t. $\| u_{\lambda}\|_{L^{2}(D)} = c$;
            \item For any $c < \eta_{2}$, \eqref{NLS} admits no solution $u_{\lambda}$ in $H_{0,rad}^{1}(D)$, s.t. $\| u_{\lambda}\|_{L^{2}(D)} = c$.
        \end{itemize}
        \item If $N = 2$ and $p > 6$, for any $c > 0$, \eqref{NLS} admits at least one solution $u_{\lambda}$ in $H_{0,rad}^{1}(D)$ with some $\lambda > 0$, s.t. $\| u_{\lambda}\|_{L^{2}(D)} = c$.
        \item If $N = 2$ and $p = 4$, there exists $\eta_{3} > 0$ such that:
        \begin{itemize}
            \item For any $c > \eta_{3}$, \eqref{NLS} admits at least one solution $u_{\lambda}$ in $H_{0,rad}^{1}(D)$ with some $\lambda > 0$, s.t. $\| u_{\lambda}\|_{L^{2}(D)} = c$; 
            \item For any $c < \eta_{3}$, \eqref{NLS} admits no solution $u_{\lambda}$ in $H_{0,rad}^{1}(D)$, s.t. $\| u_{\lambda}\|_{L^{2}(D)} = c$.
        \end{itemize}
        \item If $2 < p < 2 + \frac{4}{N}$, for any $c > 0$, \eqref{NLS} admits at least one solution $u_{\lambda}$ in $H_{0,rad}^{1}(D)$ with some $\lambda > 0$, s.t. $\| u_{\lambda}\|_{L^{2}(D)} = c$.
    \end{enumerate}
    For problems on the unit ball or $\mathbb{R}^N$, radial solutions of \eqref{NLS} have a unique maximum point, whereas maximum points of radial solutions for \eqref{NLS} on 
    $D$ form a circle, resulting in the critical exponent $p = 6$. This method is also applicable to an annulus, and the result is similar to that of the exterior of a ball due 
    to the domains have identical topological structure. 

    For any $\lambda > -\lambda_1$, \cite[Theorem 1.2]{yadava1997uniqueness}, \cite[Theorem 1.1]{yao2019uniqueness}, \cite[Theorem 1.1]{ni1985uniqueness}, \cite[Theorem 1.1]{felmer2008uniqueness} showed that 
    \eqref{NLS} possesses a unique and non-degenerate solution $u_{\lambda}$ in $H_{0,rad}^{1}(T)$, the Sobolev space of radially symmetric functions.
    In fact, $u_{\lambda}$ is the ground state of \eqref{NLS}, i.e. $\Phi_{\lambda}(u_{\lambda}) = \inf_{u \in N_{\lambda}}\Phi_{\lambda}(u)$, where $N_{\lambda}$ is the Nehari manifold with a parameter $\lambda$ defined by
    \begin{equation*}
        N_{\lambda}:= \{u \in H_{0,rad}^{1}(T) \backslash \{0\}: \int_{T}| \nabla u |^{2} + \lambda u^{2} - |u|^{p} = 0\}.
    \end{equation*}

    So for \eqref{NLS} in an annular domain, this paper addresses two primary questions:
    \begin{itemize}
        \item \textbf{Is there a positive solution in the annulus that can have an arbitrarily large mass?}
        \item \textbf{Does $\mu_k \rightarrow +\infty$ as $k \rightarrow \infty$ in an annulus?}
    \end{itemize}
    
    To address these questions, we employ the techniques proposed in \cite{song2022threshold}, which lead to the main theorem of this paper as follows:

    \begin{theorem}\label{thm1.1} \hfill
        \begin{enumerate}[leftmargin = 2em, itemindent = 0pt]
            \item If $N \geq 3$ or if $N = 2$ and $p < 6$, for any $c > 0$, \eqref{NLS} admits at least one solution $u_{\lambda}$ in $H_{0,rad}^{1}(T)$ with some $\lambda > 0$ s.t. $\| u_{\lambda}\|_{L^{2}(T)} = c$.
            \item If $N = 2$ and $p = 6$, there exists $\eta_{1} > 0$ such that:
            \begin{itemize}
                \item For any $c < \eta_{1}$, \eqref{NLS} admits at least one solution $u_{\lambda}$ in $H_{0,rad}^{1}(T)$ with some $\lambda > 0$ s.t. $\| u_{\lambda}\|_{L^{2}(T)} = c$;
                \item For any $c > \eta_{1}$, \eqref{NLS} admits no solution $u_{\lambda}$ in $H_{0,rad}^{1}(T)$ s.t. $\| u_{\lambda}\|_{L^{2}(T)} = c$.
            \end{itemize}
            \item If $N = 2$ and $p > 6$, there exists $\eta_{2} > 0$ such that:
            \begin{itemize}
                \item For any $c < \eta_{2}$, \eqref{NLS} admits at least two solutions $u_{\lambda}$ and $u_{\widetilde{\lambda}}$ in $H_{0,rad}^{1}(T)$ with $\lambda < \widetilde{\lambda} > 0$ 
                s.t. $\| u_{\lambda}\|_{L^{2}(T)} = \| u_{\widetilde{\lambda}}\|_{L^{2}(T)} = c$;
                \item \eqref{NLS} admits at least one solution $u_{\lambda}$ in $H_{0,rad}^{1}(T)$ with some $\lambda > 0$ s.t. $\| u_{\lambda}\|_{L^{2}(T)}=\eta_{2}$;
                \item For any $c > \eta_{2}$, \eqref{NLS} admits no solution $u_{\lambda}$ in $H_{0,rad}^{1}(T)$ s.t. $\| u_{\lambda}\|_{L^{2}(T)} = c$.
            \end{itemize}
        \end{enumerate}
    \end{theorem}

    Utilizing \cite[Corollary 2.5]{song2022properties}, it is known that $\{(\lambda,u_{\lambda}): \lambda > 0\}$ forms a continuous curve in $\mathbb{R} \times H_{0,rad}^{1}(T)$. To conclude the proof of Theorem \ref{thm1.1}, 
    the behavior of $\int_{T}u_{\lambda}^{2}$ will be analyzed as $\lambda$ approaches $+\infty$ and $-\lambda_1$.

    \begin{remark}
        As previously mentioned, $u_{\lambda}$ is the ground state of \eqref{NLS}, so it has Morse index one in $H_{0,rad}^{1}(T)$. \cite{esposito2007asymptotic} and \cite{sirakov2001symmetry} showed that $u_{\lambda}$ has a unique point
        (in fact, a sphere) $r_{\lambda}$ achieving the maximum and blows up when $\lambda \rightarrow +\infty$. However, the solutions obtained in \cite{pierotti2017normalized} have multiple peaks and hence cannot manifest a spherically symmetric 
        blow-up domain like $u_{\lambda}$. Thus, we may deduce that $u_{\lambda}$ must be distinct from the solutions identified in \cite{pierotti2017normalized}. Applying the blow-up analysis proposed in \cite{pierotti2017normalized} 
        (or \cite{esposito2011pointwise}), we infer that as $\lambda \rightarrow +\infty$, the Morse index of $u_{\lambda}$ in $H_{0}^{1}(T)$ approaches infinity. Therefore, we may conclude that in case (1), $\mu_k \rightarrow +\infty$ as $k \rightarrow \infty$.
    \end{remark}

    Our method yields the orbital stability/instability of standing wave solutions $e^{i\lambda t}u_{\lambda}(x)$ for \eqref{NLSE} as a byproduct. Recall that such 
    solutions are called orbitally stable if for each $\epsilon > 0$, there exists $\delta > 0$ such that, whenever $\Psi_{0} \in H_{0,rad}^{1}(T, \mathbb{C})$ is such that 
    $\|\Psi_{0} - u_{\lambda}\|_{H_{0}^{1}(T, \mathbb{C})} < \delta$ and $\Psi(t, x)$ is the solution of \eqref{NLSE} with $\Psi(0, \cdot) = \Psi_{0}$ in some interval $[0, t_{0})$, 
    then $\Psi(t, \cdot)$ can be continued to a solution in $0 \leq t < \infty$ and
    \begin{equation*}
        \sup_{0 < t < \infty}\inf_{s \in \mathbb{R}} \|\Psi(t, x) - e^{i\lambda s}u_{\lambda}(x)\|_{H_{0}^{1}(T, \mathbb{C})} < \epsilon;
    \end{equation*}
    otherwise, they are called unstable. We assume the following condition:

    $(LWP)$ For each $\Psi_{0} \in H_{0}^{1}(T, \mathbb{C})$, there exist $t_{0} > 0$, only depending on $\|\Psi_{0}\|_{H_{0}^{1}}$, and a unique solution $\Psi(t, x)$ of 
    \eqref{NLSE} with initial datum $\Psi_{0}$ in the interval $I = [0, t_{0})$.

    \begin{theorem}\label{thm1.2} \hfill
        \begin{enumerate}[leftmargin = 2em, itemindent = 0pt]
            \item If $N \geq 3 $ or if $N = 2$ and $p < 6$. Let $u_{\lambda}$ be given by Theorem \ref{thm1.1} (1), $\Psi(t, x) = e^{i\lambda t}u_{\lambda}(x)$. 
            Then for a.e. $c \in (0, +\infty)$, $\Psi$ is orbitally stable in $H_{0,rad}^{1}(T, \mathbb{C})$.
            \item If $N = 2$ and $p = 6$. Let $u_{\lambda}$ be given by Theorem \ref{thm1.1} (2), $\Psi(t, x) = e^{i\lambda t}u_{\lambda}(x)$. 
            Then for a.e. $c \in (0, \eta_1)$, $\Psi$ is orbitally stable in $H_{0,rad}^{1}(T, \mathbb{C})$.
            \item If $N = 2$ and $p < 6$. Let $u_{\lambda}$ and $u_{\widetilde{\lambda}}$ be given by Theorem \ref{thm1.1} (3), $\Psi(t, x) = e^{i\lambda t}u_{\lambda}(x)$, 
            $\widetilde{\Psi}(t, x) = e^{i\widetilde{\lambda} t}u_{\widetilde{\lambda}}(x)$. Then for a.e. $c \in (0, \eta_2)$, $\Psi$ is orbitally stable while $\widetilde{\Psi}$ is 
            orbitally unstable in $H_{0,rad}^{1}(T, \mathbb{C})$.      
        \end{enumerate}
    \end{theorem}

    \begin{remark}
        If we replace the annulus domain with a domain possessing a "nontrivial topology", such as $\Omega \ \backslash \ \Omega'$ where $\Omega$ and $\Omega' \subset \Omega$ are both simple-connected domains.
        We believe that there is a solution which has a "spherical blow-up" property, i.e. when $\lambda \rightarrow +\infty$, the solution has a domain which is homotopic to a sphere, attains a maximum in a neighborhood and blows up.
        Therefore, we speculate that the conclusions presented in \eqref{thm1.1} and \eqref{thm1.2} hold for this scenario as well. However, we currently lack a methodology to address this problem.
    \end{remark}

    The paper is organized as follows: In Section 2, the study of the asymptotic properties when $\lambda \rightarrow +\infty$ is presented. Section 3 is dedicated to the analysis 
    of the asymptotic properties when $\lambda \rightarrow -\lambda_1$, and it includes the completion of the proofs of Theorem \ref{thm1.1}. Lastly, the orbital stability/instability 
    of standing wave solutions is discussed, and Theorem \ref{thm1.2} is demonstrated in Section 4.

    \begin{notation}
        We use the standard notation $\{\varphi_k\}_{k \geq 1}$ for a basis of eigenfunctions of the Dirichlet Laplacian on $H_0^1(T)$, orthogonal in $H_0^1(T)$ 
        and orthonormal in $L^2(T)$. Such functions are ordered in such a way that the corresponding eigenvalues $\lambda_k$ satisfy $0 < \lambda_1 < \lambda_2 \leq \lambda_3 \leq \cdots$.
    \end{notation}

\section{Asymptotic behaviors when $\lambda \rightarrow +\infty$}

    \begin{lemma} \label{lem2.1}
        Let $u_{\lambda} = u_{\lambda}(r)$ be the unique solution of \eqref{NLS} in $H_{0,rad}^{1}(T)$. 
        Then $u_{\lambda}(r)$ has a unique maximum point $\bar{r}_{\lambda}$ and $u_{\lambda}' > 0$ 
        in $(1,\bar{r}_{\lambda})$, $u_{\lambda}' < 0$ in $(\bar{r}_{\lambda},2)$.
    \end{lemma}

    \begin{proof}
        By a standard elliptic estimate, we can assert that $u_{\lambda}(r) \in C^{2}((1,2))$. Furthermore, the strong maximum principle implies that $u_{\lambda}'(1) > 0$. We may assume that 
        $u_{\lambda}'(\bar{r}_{\lambda}) = 0$, and $u_{\lambda}' > 0$ in $(1,\bar{r}_{\lambda})$. Subsequently, we observe that $u_{\lambda}$ satisfies the following equation:
        \begin{equation} \label{eqa2.1}
            \begin{gathered}
            -(u'' + \frac{N-1}{r}u') + \lambda u = u^{p-1} \ in \ (\bar{r}_{\lambda},2), \\ 
            u(\bar{r}_{\lambda}) = \alpha, u'(\bar{r}_{\lambda}) = 0, u(2) = 0.
            \end{gathered}
        \end{equation}
        where $\alpha = u_{\lambda}(\bar{r}_{\lambda})$. Referring to \cite[Theorem 2]{sirakov2001symmetry}, it follows that $u'(r) < 0$ in $(\bar{r}_{\lambda},2)$.
    \end{proof}

    According to the result in \cite{dancer1997some}, the unique solution attains its maximum value on a sphere, and the radius of that sphere tends to 1.

    \begin{lemma}\label{lem2.2} 
        Let $u_{\lambda}$ be the unique solution of \eqref{NLS} in $H_{0,rad}^{1}(T)$ with $\lambda \rightarrow +\infty$, then $\liminf_{\lambda \rightarrow +\infty}\bar{r}_{\lambda} = 1$.
    \end{lemma}

    \begin{lemma}\label{lem2.3} 
        Let $u_{\lambda}$ be the unique solution of \eqref{NLS} in $H_{0,rad}^{1}(T)$. Then        
        \begin{equation}
        0 < \liminf_{\lambda \rightarrow +\infty}\frac{\lambda}{\|u_{\lambda}\|_{L^{\infty}}^{p-2}} \leq 
        \limsup_{\lambda \rightarrow +\infty}\frac{\lambda}{\|u_{\lambda}\|_{L^{\infty}}^{p-2}} \leq 1.
        \end{equation}
    \end{lemma}

    \begin{proof}
        Let $\bar{r}_{\lambda}$ be the unique maximum point of $u_{\lambda}$. Then
        \begin{equation}
            \lambda u_{\lambda}(\bar{r}_{\lambda}) < -u_{\lambda}''(\bar{r}_{\lambda}) + \lambda u_{\lambda}(\bar{r}_{\lambda}) = u_{\lambda}(\bar{r}_{\lambda})^{p-1},
        \end{equation}
        i.e. $\lambda < u_{\lambda}(\bar{r}_{\lambda})^{p-2}$.

        Next we will show that $\liminf_{\lambda \rightarrow +\infty}\lambda/\|u_{\lambda}\|_{L^{\infty}}^{p-2} > 0$. We will argue by contradiction: 
        Suppose on the contrary that there exists $\{\lambda_{n}\}$ with $\lambda_{n} \rightarrow +\infty$ such that 
        $\lambda_{n}/\|u_{\lambda_{n}}\|_{L^{\infty}}^{p-2} \rightarrow 0$. To simplify the notation, we shall denote $u_{n} = u_{\lambda_{n}}$ and 
        $\bar{r}_{n} = \bar{r}_{\lambda_{n}}$. Notice that $u_{n}$ satisfies
        \begin{equation}
            \begin{gathered}
                -(u_{n}'' + \frac{N-1}{r}u_{n}') + \lambda_{n} u_{n} = u_{n}^{p-1} \ \text{in} \ I, \\
                u_{n}(1) = u_{n}(2) = 0, 
            \end{gathered}
        \end{equation}
        where $I = (1,2)$. We consider
        \begin{equation} \label{eqa2.5}
            v_{n}(r) = \frac{1}{\|u_{n}\|_{L^{\infty}(I)}} u_{n} (\frac{r}{\|u_{n}\|_{L^{\infty}(I)}^{\frac{p-2}{2}}} + \bar{r}_{n})  
        \end{equation}
        where $r \in I_{n} = (\|u_{n}\|_{L^{\infty}(I)}^{\frac{p-2}{2}}(1 - \bar{r}_{n}),\|u_{n}\|_{L^{\infty}(I)}^{\frac{p-2}{2}}(2 - \bar{r}_{n}))$. Then $\|v_{n}\|_{L^{\infty}(I_{n})} = v_{n}(0) = 1$ and
        \begin{equation} \label{eqa2.6}
            -(v_{n}'' + \frac{N-1}{r + \|u_{n}\|_{L^{\infty}(I)}^{\frac{p-2}{2}}\bar{r}_{n}}v_{n}') + \frac{\lambda_{n}}{\|u_{n}\|_{L^{\infty}}^{p-2}}v_{n} = v_{n}^{p-1} \text{ in } I_{n}, \quad
            v_{n} = 0 \text{ on } \partial I_{n}.
        \end{equation}
        By Lemma 2.2, we have $\bar{r}_{n} \rightarrow 1$ up to a subsequence. So we have
        \begin{equation*}
            \lim_{n \rightarrow +\infty}\|u_{n}\|_{L^{\infty}(I)}^{\frac{p-2}{2}}(2 - \bar{r}_{n}) = +\infty.
        \end{equation*}
        If
        \begin{equation*}
            \lim_{n \rightarrow +\infty}\|u_{n}\|_{L^{\infty}(I)}^{\frac{p-2}{2}}(1 - \bar{r}_{n}) = -\infty,
        \end{equation*}
        $I_{n} \rightarrow \mathbb{R}$. For any $R > 0$, $(-R,R)$ is contained in $I_{n}$ for $n$ large enough. Applying standard elliptic estimates and passing to a subsequence if necessary, 
        we may assume that $v_{n} \rightarrow v$ in $C_{loc}(\mathbb{R})$ where $v(0) = 1$ and $v$ satisfies
        \begin{equation} \label{eqa2.7}
            -v'' = v^{p-1} \ in \ \mathbb{R}, 0 \leq v \leq 1 \ in \ \mathbb{R}, v \in C^{2}(\mathbb{R}).
        \end{equation}
        If
        \begin{equation*}
            \lim_{n \rightarrow +\infty}\|u_{n}\|_{L^{\infty}(I)}^{\frac{p-2}{2}}(1 - \bar{r}_{n}) = -k \leq 0,
        \end{equation*}
        $I_{n} \rightarrow (-k,+\infty)$. Similarly, we obtain a solution $v$ with $v(0) = 1$ of
        \begin{equation} \label{eqa2.8}
        -v'' = v^{p-1} \ in \ (-k,+\infty), v(-k) = 0, 0 \leq v \leq 1, v \in C^{2}([-k,+\infty)).
        \end{equation}
        Noticing that $v(0) = 1$ and $v(-k) = 0$, we know that $k > 0$. Then the following Lemma \ref{lem2.4} provides a contradiction.
    \end{proof}

    \begin{lemma}\label{lem2.4} 
        \eqref{eqa2.7} (resp. \eqref{eqa2.8}) admits no solution $v$ that satisfies $v(0) = 1$.
    \end{lemma}

    \begin{proof}
        See \cite[Lemma2.4]{song2022threshold}.
    \end{proof}

    \begin{theorem}\label{thm2.5} 
        Let $u_{n}(r) = u_{\lambda_{n}}$ be the unique solution of \eqref{NLS} in $H_{0,rad}^{1}(T)$ with $\lambda = \lambda_{n} \rightarrow +\infty$ and
        \begin{equation}
            \begin{gathered}
                \omega_{n}(r) = \lambda_{n}^{\frac{1}{2-p}} u_{n}(\frac{r}{\sqrt{\lambda_{n}}} + \bar{r}_{n}), \\
                r \in \widetilde{I}_{n} = (\sqrt{\lambda_{n}}(1 - \bar{r}_{n}),\sqrt{\lambda_{n}}(2 - \bar{r}_{n})),
            \end{gathered}
        \end{equation}
        where $\bar{r}_{n} = \bar{r}_{\lambda_{n}}$ is the unique maximum point of $u_{n}$. Passing to a subsequence if necessary, we may assume that 
        $\omega_{n}(r) \rightarrow W(r)$ in $H_{0}^{1}(\mathbb{R})$ where $W \in H^{1}(\mathbb{R})$ is the unique solution of
        \begin{equation} \label{eqa2.10}
        -W'' + W = W^{p-1} \ in \ \mathbb{R}, W(r) \rightarrow 0 \ as \ |r| \rightarrow +\infty.
        \end{equation}
    \end{theorem}

    \begin{proof}
        Note that $\omega_{n}$ satisfies
        \begin{equation} \label{eqa2.11}
            -(\omega_{n}'' + \frac{N-1}{r + \sqrt{\lambda_{n}}\bar{r}_{n}}\omega_{n}') + \omega_{n} = \omega_{n}^{p-1} \ in \ \widetilde{I}_{n}, \omega_{n} = 0 \ on \ \partial \widetilde{I}_{n}.
        \end{equation}
        By Lemma \ref{lem2.3}, $\omega_{n}(0) = \|\omega_{n}\|_{L^{\infty}}$ is bounded. Similar to Lemma \ref{lem2.3} and \ref{lem2.4}, we can deduce that
        \begin{equation*}
            \lim_{n \rightarrow +\infty}\sqrt{\lambda_{n}}(1 - \bar{r}_{n}) = -\infty, \quad \lim_{n \rightarrow +\infty}\sqrt{\lambda_{n}}(2 - \bar{r}_{n}) = +\infty.
        \end{equation*}
        So we know that $\omega_{n} \rightarrow W$ in $C_{loc}^{2}(\mathbb{R})$ where $W(0) = \max W(r)$ and $W$ satisfies
        \begin{equation}
            -W'' + W = W^{p-1} \ in \ \mathbb{R}, W \geq 0, W \in C^{2}(\mathbb{R}).
        \end{equation}
        According to Lemma \ref{lem2.1}, $W$ is increasing in $(-\infty,0)$ and decreasing in $(0,+\infty)$. Thus $W(r) \rightarrow 0$ as $r \rightarrow +\infty$, i.e. $W$ is the unique solution of \eqref{eqa2.10}.
        
        The convergence in $H_{0}^{1}(\mathbb{R})$ and the exponential decay argument can be shown as \cite[Theorem2.5]{song2022threshold}.
    \end{proof}

    \begin{corollary}
        $\omega_{n}$ and $W$ are given by Theorem \ref{thm2.5}. Passing to a subsequence if necessary, we have
        \begin{equation} \label{eqa2.12}
        \int_{\sqrt{\lambda_{n}}(1 - \bar{r}_{n})}^{\sqrt{\lambda_{n}}(2 - \bar{r}_{n})}\omega_{n}^{2}r^{k} \rightarrow \int_{-\infty}^{+\infty}W^{2}r^{k}, k = 0, 1, \cdots, N-1.
        \end{equation}
    \end{corollary}

    \begin{proof}
        \eqref{eqa2.12} directly follows by the convergence of $\omega_{n}$ in $C_{loc}^{2}(\mathbb{R})$ and the uniform exponential decay.
    \end{proof}

    The main theorem of this section is as follows:

    \begin{theorem}\label{thm2.7} 
        Let $u_{\lambda}$ be the unique solution of \eqref{NLS} in $H_{0,rad}^{1}(T)$ with $\lambda \rightarrow +\infty$.
        \begin{enumerate}
            \item When $N = 2$ and $p < 6$ or $N \geq 3$, $\int_{T}u_{\lambda}^{2} \rightarrow +\infty$.
            \item When $N = 2$ and $p = 6$, $\liminf_{\lambda \rightarrow +\infty}\int_{T}u_{\lambda}^{2} \in (0,+\infty)$.
            \item When $N = 2$ and $p > 6$,  $\int_{T}u_{\lambda}^{2} \rightarrow 0$.
        \end{enumerate}
    \end{theorem}

    \begin{proof}
        Let $u_{n} = u_{\lambda_{n}}$ be the unique solution of \eqref{NLS} in $H_{0,rad}^{1}(T)$ with $\lambda = \lambda_{n} \rightarrow +\infty$ and $\bar{r}_{n} = \bar{r}_{\lambda_{n}}$ be the unique maximum point of $u_{n}$. Direct computations imply 
        \begin{equation}\label{eqa2.13}
            \begin{aligned} 
                \int_{T}u_{n}^{2} &= C\int_{1}^{2}u_{n}^{2}r^{N-1}dr \\
                &= C\int_{1 - \bar{r}_{n}}^{2 - \bar{r}_{n}}(u_{n}(r+\bar{r}_{n}))^{2}(r+\bar{r}_{n})^{N-1}dr \\
                &= C\int_{1 - \bar{r}_{n}}^{2 - \bar{r}_{n}}(u_{n}(r+\bar{r}_{n}))^{2}(\Sigma_{k = 0}^{N-1}C_{N-1}^kr^k\bar{r}_{n}^{N-1-k})dr \\
                &= C\Sigma_{k = 1}^{N-1}C_{N-1}^k\lambda_{n}^{\frac{2}{p-2} - \frac{k+1}{2}}\bar{r}_{n}^{N-1-k}\int_{\sqrt{\lambda_{n}}(1 - \bar{r}_{n})}^{\sqrt{\lambda_{n}}(2 - \bar{r}_{n})}\omega_{n}^{2}r^{k}dr,
            \end{aligned}
        \end{equation}
        where $C = \int_{S^{N-1}}d\sigma> 0$, $C_{N-1}^k = \frac{(N-1)(N-2)\cdots(N-k)}{k(k-1)\cdots 1}$ if $k \geq 1$ and $C_{N-1}^0 = 1$.

        Proof of (1): We will argue by contradiction, suppose that there exists $\{u_{n}\}$ such that $u_{n} = u_{\lambda_{n}}$ is the unique solution of \eqref{NLS} in $H_{0,rad}^{1}(T)$ with $\lambda = \lambda_{n} \rightarrow +\infty$ and
        \begin{equation} \label{eqa.13}
            \lim_{n \rightarrow +\infty}\int_{T}u_{n}^{2} < +\infty.
        \end{equation}
        When $N = 2$, we assume that $p < 6$. When $N \geq 3$, $p < 2N/(N-2) \leq 6$. Up to a subsequence, \eqref{eqa2.13} shows that for large $n$,
        \begin{equation}
            \begin{aligned}
                \int_{T}u_{n}^{2} &\geq \frac{C}{2}\lambda_{n}^{\frac{2}{p-2} - \frac{1}{2}}\bar{r}_{n}^{N-1}\int_{\sqrt{\lambda_{n}}(1 - \bar{r}_{n})}^{\sqrt{\lambda_{n}}(2 - \bar{r}_{n})}\omega_{n}^{2}dr \\ 
                                  &\geq \frac{C}{2}\lambda_{n}^{\frac{2}{p-2} - \frac{1}{2}}\int_{\sqrt{\lambda_{n}}(1 - \bar{r}_{n})}^{\sqrt{\lambda_{n}}(2 - \bar{r}_{n})}\omega_{n}^{2}dr \rightarrow +\infty,
            \end{aligned}
        \end{equation}
        which is a contradiction with \eqref{eqa.13}.

        \medskip

        Proof of (2): First, we prove that $\liminf_{\lambda \rightarrow +\infty}\int_{T}u_{\lambda}^{2} > 0$. We will argue by contradiction. Assume that there exists $\{u_{n}\}$ such that $u_{n} = u_{\lambda_{n}}$ is the unique solution of \eqref{NLS} in $H_{0,rad}^{1}(T)$ with $\lambda = \lambda_{n} \rightarrow +\infty$ and
        \begin{equation} \label{eqa.14}
            \int_{T}u_{n}^{2} \rightarrow 0.
        \end{equation}
        Up to a subsequence, \eqref{eqa2.13} shows that for large $n$,
        \begin{equation}
            \begin{aligned}
                \int_{T}u_{n}^{2} &\geq \frac{C}{2}\bar{r}_{n}^{N-1}\int_{\sqrt{\lambda_{n}}(1 - \bar{r}_{n})}^{\sqrt{\lambda_{n}}(2 - \bar{r}_{n})}\omega_{n}^{2}dr \\
                                  &\geq \frac{C}{2}\int_{\sqrt{\lambda_{n}}(1 - \bar{r}_{n})}^{\sqrt{\lambda_{n}}(2 - \bar{r}_{n})}\omega_{n}^{2}dr \rightarrow \frac{C}{2}\int_{-\infty}^{+\infty}W^{2}dr > 0,
            \end{aligned}
        \end{equation}
        where $W$ is given by Theorem \ref{thm2.5}. Thus we find a contradiction with \eqref{eqa.14}.

        Next, we prove that $\liminf_{\lambda \rightarrow +\infty}\int_{T}u_{\lambda}^{2} < +\infty$. According to Lemma \ref{lem2.2}, up to a subsequence, \eqref{eqa2.13} shows that
        \begin{equation}
            \begin{aligned}              
                \int_{T}u_{n}^{2} &= C\Sigma_{k = 0}^{N-1}C_{N-1}^k\lambda_{n}^{\frac{1}{2} - \frac{k+1}{2}}\bar{r}_{n}^{N-1-k}\int_{\sqrt{\lambda_{n}}(1 - \bar{r}_{n})}^{\sqrt{\lambda_{n}}(2 - \bar{r}_{n})}\omega_{n}^{2}r^{k}dr \\
                &\rightarrow C\int_{-\infty}^{+\infty}W^{2}dr < +\infty,
            \end{aligned}
        \end{equation}
        where $W$ is given by Theorem \ref{thm2.5}. Consequently we arrive at the conclusion as desired.

        \medskip

        Proof of (3): We will argue by contradiction, let us assume that there exists $\{u_{n}\}$ such that $u_{n} = u_{\lambda_{n}}$ is the unique solution of \eqref{NLS} in $H_{0,rad}^{1}(T)$ with $\lambda = \lambda_{n} \rightarrow +\infty$ and
        \begin{equation} \label{eqa.16}
            \lim_{n \rightarrow +\infty}\int_{T}u_{n}^{2} > 0.
        \end{equation}
            When $N = 2$ and $p > 6$, up to a subsequence, \eqref{eqa2.13} shows that for large $n$,
        \begin{equation}
        \int_{T}u_{n}^{2} \leq 2C\lambda_{n}^{\frac{2}{p-2} - \frac{1}{2}}\bar{r}_{n}^{N-1}\int_{\sqrt{\lambda_{n}}(1 - \bar{r}_{n})}^{\sqrt{\lambda_{n}}(2 - \bar{r}_{n})}\omega_{n}^{2}dr \rightarrow 0,
        \end{equation}
        which is a contradiction with (\ref{eqa.16}).
    \end{proof}

\section{Asymptotic behaviors when $\lambda \rightarrow -\lambda_1$}

    As stated in the introduction, \eqref{NLS} possesses a unique and non-degenerate solution $u_{\lambda}$ in $H_{0,rad}^{1}(T)$ for any fixed $\lambda > -\lambda_1$. Furthermore, $u_{\lambda}$ 
    is the ground state of \eqref{NLS}. Then similar to the proof of \cite[Corollary 2.5]{song2022properties}, we know that $\{(\lambda,u_{\lambda}): \lambda > -\lambda_1\}$ is a 
    continuous curve in $\mathbb{R} \times H_{0,rad}^{1}(T)$. 
    
    Set $d(\lambda) = \int_{T}u_{\lambda}^{2}$, which is a continuous function of $\lambda > -\lambda_1$. Note that \eqref{NLS} 
    has a solution with $L^2$ norm $c$ is equivalent to that $d(\lambda) = c$ for some $\lambda > -\lambda_1$. Also, the number of normalized solutions with mass $c$ is the number of $\lambda$ in $\{\lambda > -\lambda_1: d(\lambda) = c\}$. By a standard bifurcation argument, 
    it is easy to show that $-\lambda_1$ is a bifurcation point and solution that bifurcates from $-\lambda_1$ is radial. Use the uniqueness results, we deduce that $\lim_{\lambda \rightarrow -\lambda_1} d(\lambda) = 0$.

    \medskip

    \noindent \textbf{Proof of Theorem \ref{thm1.1}:  }  

    \medskip

    Proof of (1): If $N \geq 3$ or if $N = 2 , p < 6$, we have $\lim_{\lambda \to -\lambda_1} d(\lambda) = 0$, $\lim_{\lambda \to +\infty} d(\lambda) = +\infty$.
    Hence, for any $c > 0$, $d^{-1}(c)$ is not empty. Then the conclusion of Theorem \ref{thm1.1} (1) holds true.

    \medskip

    Proof of (2): If $N = 2 , p = 6$, we have $\lim_{\lambda \to -\lambda_1} d(\lambda) = 0$, $\liminf_{\lambda \to +\infty} d(\lambda) \in (0,+\infty)$. If we set
    $\eta_1 = \sup_{\lambda > -\lambda_1} d(\lambda)$. If $\sup_{\lambda > 0}d(\lambda) = \limsup_{\lambda \rightarrow +\infty}d(\lambda)$, then $\eta_1 = \limsup_{\lambda \rightarrow +\infty}d(\lambda)$. Otherwise,
    $\eta_1 = \max_{\lambda > 0}d(\lambda) > \limsup_{\lambda \rightarrow +\infty}d(\lambda)$. Hence, for any $c \in (0,\eta_1)$, $d^{-1}(c)$ is not empty. 
    And for any $c > \eta_1$, $d^{-1}(c)$ is empty. Then we can get the results in Theorem \ref{thm1.1} (2).

    \medskip

    Proof of (3): If $N = 2 , p > 6$, we have $\lim_{\lambda \to -\lambda_1} d(\lambda) = 0$, $\lim_{\lambda \to +\infty} d(\lambda) = 0$.
    Hence, there exists some $\hat{\lambda}$ such that $d(\hat{\lambda}) = \sup_{\lambda > 0}d(\lambda) = \max_{\lambda > 0}d(\lambda) > 0$ since $d(\lambda)$ is continuous. Set
    $\eta_2 = d(\hat{\lambda}) = \sup_{\lambda > 0}d(\lambda) > 0$, then for any $c \in (0,\eta_2)$, there exist $\lambda > \hat{\lambda}$ and $\widetilde{\lambda} \in (0,\hat{\lambda})$ such that $d(\lambda) = d(\widetilde{\lambda}) = c$,
    and for any $c > \eta_2$, $d^{-1}(c)$ is empty. This implies the results in Theorem \ref{thm1.1} (3).
    
\section{Orbital stability/instability results}

    \begin{lemma}\label{lem4.1} 
        $\{(\lambda,u_{\lambda}): \lambda > 0\}$ is a $C^{1}$ curve in $\mathbb{R} \times H_{0}^{1}(T)$.
    \end{lemma}

    \begin{proof}
        By \cite[Theorem 1.2]{yadava1997uniqueness}, \cite[Theorem 1.1]{yao2019uniqueness}, \cite[Theorem 1.1]{ni1985uniqueness}, \cite[Theorem 1.1]{felmer2008uniqueness}, \eqref{NLS} possesses a unique and non-degenerate solution $u_{\lambda}$ 
        in $H_{0,rad}^{1}(T)$ for any fixed $\lambda > -\lambda_1$. Then similar to \cite[Lemma4.1]{song2022threshold}, we can prove that $\{(\lambda,u_{\lambda}): \lambda > 0\}$ is a $C^{1}$ curve in $\mathbb{R} \times H_{0}^{1}(T)$ by the implicit function theorem.
    \end{proof}

    To study the orbital stability, we use the following result, which is stated in \cite{song2022threshold}.

    \begin{proposition}\label{prop4.2} 
        Assume that $(LWP)$ holds. Then if $\frac{d}{d\lambda}\int_{T}u_{\lambda}^{2} > 0$ (respectively $< 0$), the standing wave $e^{i\lambda t}u_{\lambda}(x)$ is orbitally stable (respectively unstable) in $H_{0,rad}^{1}(T,\mathbb{C})$.
    \end{proposition}

    \noindent \textbf{Proof of Theorem \ref{thm1.2}:  } We will prove (1) and after similar arguments, we can prove (2), (3). 
    
    First, we assume that $p > 2$ if $N \geq 3$ or $2 < p < 6$ if $N = 2$. 
    Set $d(\lambda) = \int_{T}u_{\lambda}^{2}$. Note that $\lim_{\lambda \to -\lambda_1} d(\lambda) = 0$ , $\lim_{\lambda \to +\infty} d(\lambda) = +\infty$.
    Take $c \in (0,+\infty)$ a regular value of $d$. Assume that $d^{-1}(c) = \{\lambda_{c}^{(1)}, \lambda_{c}^{(2)}, \cdots\}$ with $0< \lambda_{c}^{(1)} < \cdots$, 
    then $d'(\lambda_{c}^{(i)}) \neq 0$, $i = 1, 2, \cdots$. Noticing that $d(\lambda) < d(\lambda_{c}^{(1)})$ in $(0,\lambda_{c}^{(1)})$, we have $d'(\lambda_{c}^{(1)}) > 0$. 
    Take $\lambda = \lambda_{c}^{(1)}$ then according to Proposition \eqref{prop4.2}, the standing wave $e^{i\lambda t}u_{\lambda}(x)$ is orbitally stable in $H_{0,rad}^{1}(T,\mathbb{C})$. 
    By Sard's Theorem, regular values of $d$ are almost every $c \in (0, +\infty)$ and thus we complete the proof of (1).

    \bigskip

    \noindent \textbf{Funding:} No funding was received for conducting this study.

    \bigskip

    \noindent \textbf{Author Contribution:} Jian Liang wrote the manuscript's main text and Linjie Song gave many useful ideas, suggestions and information about this manuscript. All authors reviewed the manuscript.

    \bigskip

    \noindent \textbf{Conflict of Interest:} The authors declare that they have no conflict of interests.

    \bigskip

    \noindent \textbf{Data Availability Statement:} My manuscript has no associate data.

    \bibliography{Reference}

\providecommand{\bysame}{\leavevmode\hbox to3em{\hrulefill}\thinspace}
\providecommand{\MR}{\relax\ifhmode\unskip\space\fi MR }
\providecommand{\MRhref}[2]{%
  \href{http://www.ams.org/mathscinet-getitem?mr=#1}{#2}
}
\providecommand{\href}[2]{#2}
\begin{thebibliography}{10}

\bibitem{bartsch2013normalized}
Thomas Bartsch and S{\'e}bastien de~Valeriola, \emph{Normalized solutions of nonlinear {S}chr{\"o}dinger equations}, Archiv der Mathematik \textbf{100} (2013), no.~1, 75--83.

\bibitem{bartsch2021normalized}
Thomas Bartsch, Riccardo Molle, Matteo Rizzi, and Gianmaria Verzini, \emph{Normalized solutions of mass supercritical {S}chr{\"o}dinger equations with potential}, Communications in Partial Differential Equations \textbf{46} (2021), no.~9, 1729--1756.

\bibitem{bartsch2017natural}
Thomas Bartsch and Nicola Soave, \emph{A natural constraint approach to normalized solutions of nonlinear {S}chr{\"o}dinger equations and systems}, Journal of Functional Analysis \textbf{272} (2017), no.~12, 4998--5037.

\bibitem{berestycki1983nonlinear}
Henri Berestycki and P~L Lions, \emph{Nonlinear scalar field equations, i existence of a ground state}, Archive for Rational Mechanics and Analysis \textbf{82} (1983), 313--345.

\bibitem{bieganowski2021normalized}
Bartosz Bieganowski and Jaros{\l}aw Mederski, \emph{Normalized ground states of the nonlinear {S}chr{\"o}dinger equation with at least mass critical growth}, Journal of Functional Analysis \textbf{280} (2021), no.~11, 108989.

\bibitem{cazenave1982orbital}
Thierry Cazenave and Pierre-Louis Lions, \emph{Orbital stability of standing waves for some nonlinear {S}chr{\"o}dinger equations}, Communications in Mathematical Physics \textbf{85} (1982), 549--561.

\bibitem{dancer1997some}
EN~Dancer, \emph{Some singularly perturbed problems on annuli and a counterexample to a problem of {G}idas, {N}i and {N}irenberg}, Bulletin of the London Mathematical Society \textbf{29} (1997), no.~3, 322--326.

\bibitem{esposito2007asymptotic}
P~Esposito, G~Mancini, Sanjiban Santra, and PN~Srikanth, \emph{Asymptotic behavior of radial solutions for a semilinear elliptic problem on an annulus through {M}orse index}, Journal of Differential Equations \textbf{239} (2007), no.~1, 1--15.

\bibitem{esposito2011pointwise}
Pierpaolo Esposito and Maristella Petralla, \emph{Pointwise blow-up phenomena for a {D}irichlet problem}, Communications in Partial Differential Equations \textbf{36} (2011), no.~9, 1654--1682.

\bibitem{felmer2008uniqueness}
Patricio Felmer, Salom{\'e} Mart{\'\i}nez, and Kazunaga Tanaka, \emph{Uniqueness of radially symmetric positive solutions for {$-\Delta u + u = u^p$} in an annulus}, Journal of Differential Equations \textbf{245} (2008), no.~5, 1198--1209.

\bibitem{glassey1977blowing}
Robert~T Glassey, \emph{On the blowing up of solutions to the {C}auchy problem for nonlinear {S}chr{\"o}dinger equations}, Journal of Mathematical Physics \textbf{18} (1977), no.~9, 1794--1797.

\bibitem{hajaiej2022general}
Hichem Hajaiej and Linjie Song, \emph{A general and unified method to prove the existence of normalized solutions and some applications}, arXiv preprint arXiv:2208.11862 (2022).

\bibitem{hajaiej2023strict}
\bysame, \emph{Strict monotonicity of the global branch of solutions in the {$L^2$} norm and uniqueness of the normalized ground states for various classes of {PDE}s: Two general methods with some examples}, arXiv preprint arXiv:2302.09681 (2023).

\bibitem{jeanjean1997existence}
Louis Jeanjean, \emph{Existence of solutions with prescribed norm for semilinear elliptic equations}, Nonlinear Analysis: Theory, Methods \& Applications \textbf{28} (1997), no.~10, 1633--1659.

\bibitem{ni1985uniqueness}
Wei-Ming Ni and Roger~D Nussbaum, \emph{Uniqueness and nonuniqueness for positive radial solutions of {$\Delta u + f(u, r)= 0$}}, Communications on Pure and Applied Mathematics \textbf{38} (1985), no.~1, 67--108.

\bibitem{noris2015existence}
Benedetta Noris, Hugo Tavares, and Gianmaria Verzini, \emph{Existence and orbital stability of the ground states with prescribed mass for the {$L^2$}-critical and supercritical {NLS} on bounded domains}, Analysis \& PDE \textbf{7} (2015), no.~8, 1807--1838.

\bibitem{noris2019normalized}
\bysame, \emph{Normalized solutions for nonlinear {S}chr{\"o}dinger systems on bounded domains}, Nonlinearity \textbf{32} (2019), no.~3, 1044.

\bibitem{pierotti2017normalized}
Dario Pierotti and Gianmaria Verzini, \emph{Normalized bound states for the nonlinear {S}chr{\"o}dinger equation in bounded domains}, Calculus of Variations and Partial Differential Equations \textbf{56} (2017), 1--27.

\bibitem{sirakov2001symmetry}
Boyan Sirakov, \emph{Symmetry for exterior elliptic problems and two conjectures in potential theory}, Annales de l'Institut Henri Poincar{\'e} C, Analyse non lin{\'e}aire \textbf{18} (2001), no.~2, 135--156.

\bibitem{soave2020normalized}
Nicola Soave, \emph{Normalized ground states for the {NLS} equation with combined nonlinearities}, Journal of Differential Equations \textbf{269} (2020), no.~9, 6941--6987.

\bibitem{soave2020normalizedcritical}
\bysame, \emph{Normalized ground states for the {NLS} equation with combined nonlinearities: the {S}obolev critical case}, Journal of Functional Analysis \textbf{279} (2020), no.~6, 108610.

\bibitem{song2022existence}
Linjie Song, \emph{Existence and orbital stability of the ground states with prescribed mass for the {$L^2$}-supercritical {NLS} in bounded domains and exterior domains}, Under review (2022).

\bibitem{song2022properties}
\bysame, \emph{Properties of the least action level, bifurcation phenomena and the existence of normalized solutions for a family of semi-linear elliptic equations without the hypothesis of autonomy}, Journal of Differential Equations \textbf{315} (2022), 179--199.

\bibitem{song2022new}
Linjie Song and Hichem Hajaiej, \emph{A new method to prove the existence, non-existence, multiplicity, uniqueness, and orbital stability/instability of standing waves for {NLS} with partial confinement}, arXiv preprint arXiv:2211.10058 (2022).

\bibitem{song2022threshold}
\bysame, \emph{Threshold for existence, non-existence and multiplicity of positive solutions with prescribed mass for an {NLS} with a pure power nonlinearity in the exterior of a ball}, arXiv preprint arXiv:2209.06665 (2022).

\bibitem{yadava1997uniqueness}
Shyam~L Yadava, \emph{Uniqueness of positive radial solutions of the {D}irichlet problems {$-\Delta u = u^p \pm u^q$} in an annulus}, J. Differential Equations \textbf{139} (1997), no.~1, 194--217.

\bibitem{yao2019uniqueness}
Ruofei Yao, Yi~Li, and Hongbin Chen, \emph{Uniqueness of positive radial solutions of a semilinear elliptic equation in an annulus}, Discrete \& Continuous Dynamical Systems \textbf{39} (2019), no.~3, 1585.

\end{thebibliography}

\end{document}